\documentclass[12pt]{amsproc}
  \usepackage{latexsym} 
  \usepackage[all]{xy}
  \usepackage{amsfonts} 
  \usepackage{amsthm} 
  \usepackage{amsmath} 
  \usepackage{amssymb}
  \usepackage{pifont}  
  \usepackage{enumerate}
  \xyoption{2cell}

  \def\su#1{{\sp{[#1]}}}

  \def\<{{\langle}} 
  \def\>{{\rangle}}

  \def\eps{\varepsilon}

  \def\note#1{{}} 
 \def\can{{\rm \textsf{can}}}

  \def\note#1{} 

    \def\cL{{\mathcal L}}

   \def\cK{{\mathcal K}}  
  \def\cO{{\mathcal O}}
   
  \def\cV{{\mathfrak V}}

  \def\rend#1#2{{{\rm End}\sb{#1}(#2)}}

  \def\beq{\begin{equation}} 
  \def\eeq{\end{equation}}

  \def\id{\mathrm{id}} 
  \def\im{{\rm Im}}

  \def\ot{{\otimes}}

 \def\coker{\mathrm{coker}}

 \def\WP{\mathbb{WP}}

  \newcounter{zlist}

  \newcounter{blist} 
  \newenvironment{blist}{\begin{list}{(\alph{blist})}{ 
  \usecounter{blist}\leftmargin2.5em\labelwidth2em\labelsep0.5em 
  \topsep0.6ex %\itemsep0.3ex plus0.2ex minus0.3ex 
  \parsep0.3ex plus0.2ex minus0.1ex}}{\end{list}} 

  \newcounter{rlist} 
  \newenvironment{rlist}{\begin{list}{(\roman{rlist})}{ 
  \usecounter{rlist}\leftmargin2.5em\labelwidth2em\labelsep0.5em 
  \topsep0.6ex %\itemsep0.3ex plus0.2ex minus0.3ex 
  \parsep0.3ex plus0.2ex minus0.1ex}}{\end{list}}

\def\stac#1{\raise-.2cm\hbox{$\stackrel{\displaystyle\otimes}{\scriptscriptstyle{#1}}$}}

\def\cten#1{\raise-.2cm\hbox{$\stackrel{\displaystyle\widehat{\otimes}}
{\scriptscriptstyle{#1}}$}}

  \def\Label#1{\label{#1}\ifmmode\llap{[#1] }\else 
  \marginpar{\smash{\hbox{\tiny [#1]}}}\fi} 
  \def\Label{\label}

  \newtheorem{proposition}{Proposition}[section]
  \newtheorem{lemma}[proposition]{Lemma} 
  \newtheorem{corollary}[proposition]{Corollary} 
  \newtheorem{theorem}[proposition]{Theorem} 

  \theoremstyle{definition} 
  \newtheorem{definition}[proposition]{Definition}

  \theoremstyle{remark}

  \newcounter{c} 
   
  \newcommand{\etyk}[1]{\vspace{-7.4mm}$$\begin{equation}\Label{#1} 
  \addtocounter{c}{1}} 
  \renewcommand{\]}{\ifnum \value{c}=1 $$\else \end{equation}\fi} 
\def\ot{\otimes}

\def\CC{{\mathbb C}}

\def\NN{{\mathbb N}}

\def\PP{{\mathbb P}}

\def\RR{{\mathbb R}}

\def\WW{{\mathbb W}}

\def\ZZ{{\mathbb Z}}

\newcommand{\Cc}{\mathcal{C}}

\def\spec{{\mathrm{sp}}}

\def\*C{{}^*\hspace*{-1pt}{\Cc}}

\def\text#1{{\rm {\rm #1}}}

 \def\1{\mathbf{1}}

      \def\tr{\mathrm{Tr}\ }

\begin{document} 

\title{Quantum teardrops} 
 \author{Tomasz Brzezi\'nski}
 \address{ Department of Mathematics, Swansea University, 
  Singleton Park, \newline\indent  Swansea SA2 8PP, U.K.} 
  \email{T.Brzezinski@swansea.ac.uk}   
\author{Simon A.\ Fairfax}
% \address{ Department of Mathematics, Swansea University, 
%  Singleton Park, \newline\indent  Swansea SA2 8PP, U.K.} 
\email{201102@swansea.ac.uk}

    \date{July 2011} 
  \begin{abstract} 
Algebras of functions on  quantum weighted projective spaces are introduced, and the structure of quantum weighted projective lines or  quantum teardrops are described in detail. In particular the presentation of the coordinate algebra of the quantum teardrop in terms of generators and relations and classification of irreducible $*$-representations are derived. The algebras are then analysed from the point of view of Hopf-Galois theory or the theory of quantum principal bundles. Fredholm modules and associated traces are constructed. $C^*$-algebras of continuous functions on quantum weighted projective lines are described   and their  $K$-groups computed.
  \end{abstract} 
  \maketitle

\section{Introduction}
Weighted projective spaces are examples of orbifolds that are not global quotients of manifolds by  groups. Perhaps the most elementary definition of a weighted projective space is this. Consider  the $2n+1$-dimensional sphere $S^{2n+1}$ as a subset of  $\CC^{n+1}$ of all points $(z_0,z_1,\ldots, z_{n})$ such that $\sum_i |z_i|^2 =1$. Given $n+1$ pairwise coprime numbers $l_0,\ldots l_n$, one can define the action of the group $U(1)$ of all complex numbers $u$ of modulus one on $S ^{2n+1}$ by  $(z_0,z_1,\ldots, z_{n})\cdot u = (u^{l_0}z_0,u^{l_1}z_1,\ldots, u^{l_n}z_{n})$. The weighted projective space $\WP(l_0,l_1,\ldots, l_n)$ is the quotient of $S ^{2n+1}$ by this action. For $n=1$, $\WP(1,1)$ is the two-sphere $S^2 = \CC \PP^1$, while $\WP(1,l)$ is the teardrop orbifold studied by Thurston \cite{Thu:thr}.

This formulation of weighted projective spaces makes the definition of {\em quantum weighted projective spaces}\footnote{We choose the adjective quantum rather than noncommutative to indicate that  weighted projective spaces discussed here appear in the realm of compact quantum groups and to differentiate them from noncommutative weighted projective spaces defined within algebraic geometry in \cite{AurKat:mir}.} particularly straightforward. The algebra $\cO(S^{2n+1}_q)$ of coordinate functions on the quantum sphere is the unital complex $*$-algebra with generators $z_0,z_1,\ldots ,z_n$ subject to the following relations:
\begin{subequations}\label{sph}
\begin{gather}
z_iz_j = qz_jz_i \quad \mbox{for $i<j$}, \qquad z_iz^*_j = qz_j^*z_i \quad \mbox{for $i\neq j$}, \label{sph1}\\
z_iz_i^* = z_i^*z_i + (q^{-2}-1)\sum_{m=i+1}^n z_mz_m^*, \qquad \sum_{m=0}^n z_mz_m^*=1, \label{sph2}
\end{gather}
\end{subequations}
where $q$ is a real number, $q\in (0,1)$; see \cite{VakSob:alg}. For any $n+1$ pairwise coprime numbers $l_0,\ldots, l_n$, one can define the coaction of the Hopf algebra $\cO(U(1)) = \CC [u,u^*]$, where $u$ is a unitary and grouplike generator, as
\begin{equation}\label{coaction}
\varrho_{l_0,\ldots, l_n} : \cO(S^{2n+1}_q) \to \cO(S^{2n+1}_q) \ot \CC [u,u^*] ,\qquad z_i \mapsto z_i\ot u^{l_i}, \qquad i=0,1,\ldots , n.
\end{equation}
This coaction is then extended to the whole of $\cO(S^{2n+1}_q)$ so that $\cO(S^{2n+1}_q)$ is a right $\CC [u,u^*]$-comodule algebra. The algebra 
of coordinate functions on the quantum weighted projective space is now defined as the subalgebra of $\cO(S^{2n+1}_q)$ containing all elements invariant under the coaction $\varrho_{l_0,\ldots, l_n}$, i.e.\
$$
\cO(\WP_q(l_0,l_1,\ldots, l_n)) = \cO(S^{2n+1}_q)^{co\CC [u,u^*]}:= \{x\in \cO(S^{2n+1}_q) \; |\; \varrho_{l_0,\ldots, l_n}(x) = x\ot 1\}.
$$
In the case $l_0=l_1 = \cdots = 1$ one obtains the algebra of functions on the quantum complex projective space $\CC\PP_q^{n}$ \cite{VakSob:alg}  (see also \cite{Mey:pro}).

 In this article we concentrate on the quantum weighted projective lines, i.e.\ on the case $n=1$. For a pair of coprime positive integers $k,l$, we give the presentation of $\cO(\WP_q(k,l))$ in terms of generators and relations and  classify all irreducible representations of   $\cO(\WP_q(k,l))$ (up to unitary equivalence). In particular we prove that all infinite dimensional irreducible representations are faithful. We then proceed to analyse the structure of $\cO(\WP_q(k,l))$ as  coinvariant subalgebras. We prove that $\cO(S^{3}_q)$ is a faithfully flat Hopf-Galois $\CC [u,u^*]$-extension of $\cO(\WP_q(k,l))$ or $\cO(S^{3}_q)$ is a {\em principal $\CC [u,u^*]$-comodule algebra}  with coaction $\varrho_{k,l}$ if and only if $k=l=1$. This is in perfect agreement with the classical situation where it is known that the teardrop manifolds are not global quotients of the 3-sphere. On the other hand, we prove that in the case $k=1$, $\cO(\WP_q(1,l))$ is a coinvariant subalgebra (or a base) of a principal $\CC [u,u^*]$-comodule algebra that can be identified with the coordinate algebra of the quantum lens space $\cO(L_q(l;1,l))$. We explicitly construct a suitable strong connection on $\cO(L_q(l;1,l))$ and show that $\cO(L_q(l;1,l))$ is not a cleft principal $\CC [u,u^*]$-comodule algebra (geometrically, $L_q(l;1,l)$ is not a trivial principal $U(1)$-bundle over $\WP_q(1,l)$) in two different ways. First, purely algebraically we deduce non-cleftness by studying units in the algebra $\cO(L_q(l;1,l))$. The second method involves index calculations which start with construction of Fredholm modules and associated cyclic cocycles or Chern characters $\tau_s$ on $\cO(\WP_q(k,l))$. Then, using the explicit description of strong connections in $\cO(L_q(l;1,l))$ we calculate a part of the Chern-Galois character. Finally we evaluate $\tau_s$ at the computed part of the Chern-Galois  character and show that the results are different from zero. From this we conclude that the finitely generated projective $\cO(\WP_q(1,l))$-module $\cL[1]$  (associated to $\cO(L_q(l;1,l))$) is not free, thus the principal $\CC [u,u^*]$-comodule algebra $\cO(L_q(l;1,l))$ is not cleft. Finally, we look at the topological aspects of quantum teardrops, construct  $C^*$-algebras $C(\WP_q(k,l))$ of continuous functions on  the quantum weighted projective lines and identify them as direct sums of compact operators on  (separable) Hilbert spaces with adjoined identity. Through this identification we immediately deduce the $K$-groups of $C(\WP_q(k,l))$.

Throughout we work with involutive algebras over the field of complex numbers (but the algebraic results remain true for all fields of characteristic 0). All algebras are associative and have identity.

\section{The coordinate algebra of  the quantum teardrop and its representations}\label{sec.alg} 
\setcounter{equation}{0}
The algebra of coordinate functions on the quantum three-sphere $\cO(S^{3}_q)$ is the same as the algebra of coordinate functions on the quantum group $SU(2)$, i.e.\ $\cO(S^{3}_q)=\cO(SU_q(2))$; see \cite{Wor:com}. The generators of the latter are traditionally denoted by $\alpha=z_0$ and $\beta=z_1^*$. In terms of $\alpha$ and $\beta$ relations \eqref{sph} come out as
\begin{subequations}\label{su}
\begin{gather}
\alpha\beta = q\beta\alpha, \qquad \alpha\beta^* = q\beta^*\alpha, \qquad \beta\beta^* = \beta^*\beta, \label{su.1}\\
 \alpha\alpha^* = \alpha^*\alpha +(q^{-2}-1)\beta\beta^*, \qquad
\alpha\alpha^* + \beta\beta^* =1, \label{su.2}
\end{gather}
\end{subequations}
where $q\in(0,1)$. Setting $k=l_0$ and $l=l_1$, the coaction $\varrho_{k,l}$ of $\CC [u,u^*]$ on $\cO(S^{3}_q)$ takes the form
\begin{equation}\label{rhokl}
\alpha \mapsto \alpha\ot u^k, \qquad \beta\mapsto \beta\ot {u^*}^l = \beta\ot u^{-l}, 
\end{equation}
and $\cO(\WP_q(k,l))$ is the coinvariant subalgebra of $\cO(S^{3}_q)$.

\begin{theorem}\label{thm.coinv}
(i) 
The algebra $\cO(\WP_q(k,l))$ is  the $*$-subalgebra of $\cO(S^{3}_q)$ generated by $a= \beta\beta^*$ and $b = \alpha^l\beta^k$. 

(ii) The elements $a$ and $b$ satisfy the following relations
\begin{equation}\label{wpkl}
a^*=a,\quad ab = q^{-2l}ba, \quad bb^* = q^{2kl}a^k\prod_{m=0}^{l-1}(1-q^{2m}a), \quad b^*b = a^k \prod_{m=1}^l(1-q^{-2m}a).
\end{equation}

(iii)  $\cO(\WP_q(k,l))$ is isomorphic to the $*$-algebra generated by generators $a$, $b$ and relations \eqref{wpkl}.
\end{theorem}
\begin{proof}
(i) A basis for the space $\cO(S^{3}_q)$ consists of all elements of the form $\alpha^p\beta^r\beta^{*s}$ and $\alpha^{*p}\beta^r\beta^{*s}$, $p,r,s\in \NN$. Since the coaction is an algebra map,
$$
\varrho_{k,l}(\alpha^p\beta^r\beta^{*s}) = \alpha^p\beta^r\beta^{*s}\ot u^{kp-lr+ls}, \qquad \varrho_{k,l}(\alpha^{*p}\beta^r\beta^{*s}) = \alpha^{*p}\beta^r\beta^{*s}\ot u^{-kp-lr+ls}.
$$
The first of these elements is coinvariant provided $kp-lr+ls = 0$, i.e.\ $kp = l(s-r)$. Since $k$ and $l$ are coprime numbers, $p$ must be divisible by $l$, i.e.\ $p=ml$ for some $m\in \NN$. Therefore, $r= mk +s$ and the only coinvariant elements among the $\alpha^p\beta^r\beta^{*s}$ are those of the form 
$$
\alpha^{ml}\beta^{mk}\beta^s\beta^{*s} \sim (\alpha^l\beta^k)^m(\beta\beta^{*})^s.
$$
By similar arguments, the only coinvariant elements among terms of the form $\alpha^{*p}\beta^r\beta^{*s}$ are scalar multiples of 
$
(\alpha^l\beta^k)^{*m}(\beta\beta^{*})^r.
$ We conclude that $\cO(\WP_q(k,l))$ is a subalgebra of $\cO(S^{3}_q)$ generated by a self-adjoint  element $a= \beta\beta^{*}$ and by $b=\alpha^l\beta^k$. 

(ii) An elementary calculation that uses equations \eqref{su} confirms that \eqref{wpkl} are indeed relations that $a$ and $b$ satisfy.

(iii) Denote by $\widetilde{\cO(\WP_q(k,l))}$ the unital $*$-algebra generated by generators $a$ and $b$ and relations \eqref{wpkl}. Parts (i) and (ii)  establish the existence of the following surjective $*$-algebra map 
\begin{equation}\label{theta}
\theta: \widetilde{\cO(\WP_q(k,l))}\to \cO(\WP_q(k,l)), \qquad a\mapsto \beta\beta^*, \quad b\mapsto \alpha^l\beta^k.
\end{equation}
Note that the Diamond Lemma immediately implies that the elements
\begin{equation}\label{basis}
a^mb^n,\ a^{m}b^{*n'}, \qquad m,n\in \NN, \ n'\in \NN\setminus\{0\},
\end{equation}
form a basis for $\widetilde{\cO(\WP_q(k,l))}$. The map $\theta$ sends these elements to
$$
 (\beta\beta^*)^m(\alpha^l\beta^k)^n \sim \alpha^{ln}\beta^{k(m+n)}{\beta^*}^m,\qquad 
 (\beta\beta^*)^m({\beta^*}^k{\alpha^*}^l)^{n'} \sim {\alpha^*}^{ln'}\beta^m{\beta^*}^{k(m+n')},
$$
respectively. As these are linearly independent elements of $\cO(S^{3}_q)$, the map $\theta$ is injective, hence an isomorphism of $*$-algebras as required.
\end{proof}

In the remainder of this section we study representations of $\cO(\WP_q(k,l))$ identified with the $*$-algebra $\widetilde{\cO(\WP_q(k,l))}$ generated by $a$ and $b$ subject to  relations \eqref{wpkl}. This analysis leads to an alternative proof that the map $\theta$ \eqref{theta} is an isomorphism. 

\begin{proposition}\label{prop.rep}
Up to a unitary equivalence, the following is the list of all bounded irreducible $*$-representations of $\cO(\WP_q(k,l))$.
\begin{rlist}
\item One dimensional representation
\begin{equation}\label{rep1}
\pi_0: a\mapsto 0, \qquad b\mapsto 0.
\end{equation}
\item Infinite dimensional representations $\pi_s: \cO(\WP_q(k,l))\to \rend{}{V_s}$, labelled by $s = 1, {2},  \ldots, {l}$. For each $s$, the separable Hilbert space $V_s\simeq l^2(\NN)$ has orthonormal basis $e_p^s$, $p\in \NN$, and
\begin{equation}\label{reps}
\pi_s(a) e_p^s = q^{2(lp+s)} e_p^s, \quad \pi_s(b) e_p^s = q^{k(lp+s)} \prod_{r=1}^{l}\left(1 - q^{2(lp+s-r)}\right)^{1/2}e_{p-1}^s, \quad \pi_s(b) e_0^s =0.
\end{equation}
\end{rlist}
\end{proposition}
\begin{proof}
First consider the case when $\pi(a)=0$. The relation $b^* b=a^k \prod_{m=1}^{l} \mathcal(1-q^{-2m}a)$ implies that  $\pi(b)=0$, and  this is the one dimensional representation. 

Let V denote the irreducible representation space in which $\pi(a) \neq 0$. It is immediate from the relation $ab=q^{-2l}ba$ that $\ker(\pi(a))=V$ or $\ker(\pi(a))=0$. Since the first case is excluded by assumption $\pi(a) \neq 0$, we conclude $\ker(\pi(a))=0$. Suppose that the spectrum of $\pi(a)$ is discrete and consists only of 
$0,q^2,q^4,...,q^{2l}$. If $v$ is an eigenvector of $\pi(a)$ with eigenvalue $q^{2i}$, then the relation $ab=q^{-2l}ba$ implies that 
$w=\pi(b)v$ is an eigenvector with eigenvalue $q^{2i-2l}$.
Therefore the spectrum also contains $q^{2i-2l}$, which contradicts the assumption that  $0,q^2,...,q^{2l} $ are  the only elements of the spectrum of $\pi(a)$. Thus there must exist  $\lambda \in \spec(\pi(a))$ such that $\lambda \neq q^{2i}$ for $i \in \{1,2,...,l \}$. This means that there exists a sequence $(\xi_n)_{n \epsilon \NN}$ of unit vectors in the representation space V such that 
\begin{equation} \label{limit}
\lim_{n \to \infty} \lVert \pi(a)\xi_n-\lambda \xi_n \rVert = 0.
\end{equation}
We show that there exists $N\in \NN$ and $C>0$ such that $\lVert \pi(b^*)\xi_n \rVert \geq C$, for all $n \geq \NN$, using the relation $b^* b=a^k \prod_{m=1}^{l} \mathcal(1-q^{-2m}a)$. By the remainder theorem, this relation can be expressed in the following format:
$$
b^* b=a^k \prod_{m=1}^{l} \mathcal(1-q^{-2m} a) = (a-\lambda ) p(a)+\lambda^k \prod_{m=1}^{l} \mathcal (1-q^{-2m} \lambda),
$$
where $p(a)$ is a polynomial in variable $a$ of degree less than $k+l$.
The triangle inequality and the norm property $\lVert x\rVert  \lVert y\rVert \geq \lVert xy \rVert$ imply that:
\begin{equation}
\begin{split}
\lVert\pi(a^k) \prod_{m=1}^{l} \mathcal (1-q^{-2m} \pi(a))\xi_n \rVert 
&\geq\lVert \lambda^k \prod_{m=1}^{l} \mathcal (1-q^{-2m} \lambda) \rVert 
- \lVert p(a) (\pi(a)-\lambda I)\xi_n \rVert \notag \\ 
&\geq \lVert \lambda^k \prod_{m=1}^{l} \mathcal (1-q^{-2m} \lambda) \rVert 
- \lVert p(a) \rVert \rVert(\pi(a)-\lambda I)\xi_n \rVert  \notag.
\end{split}
\end{equation}
Therefore, 
\begin{equation} 
\begin{split}
\lVert \pi(b^*) \rVert \lVert \pi(b)\xi_n \rVert 
&\geq \lVert \lambda^k \prod_{m=1}^{l} \mathcal (1-q^{-2m} \lambda) \rVert 
- \lVert p(a)\rVert \rVert(\pi(a)-\lambda I)\xi_n \rVert , \notag 
\end{split}
\end{equation}
so that 
\begin{equation}
\begin{split}
\lVert \pi(b)\xi_n \rVert 
&\geq \frac{\lVert \lambda^k \prod_{m=1}^{l} \mathcal (1-q^{-2m} \lambda) \rVert}{\lVert \pi(b^*) \rVert} -\frac{\lVert p(a)\rVert }{\lVert \pi(b^*) \rVert} \lVert (\pi(a) - \lambda I)\xi_n \rVert. \notag
\end{split}
\end{equation} 
Since $\lVert \lambda^k \prod_{m=1}^{l} \mathcal (1-q^{-2m} \lambda) \rVert/ \lVert \pi(b^*) \rVert$ is positive, the existence of the desired $N$ and $C$ follows from \eqref{limit} above. Hence we conclude that 
$$ 
\eta_n:=\frac{\pi(b)\xi_n}{\lVert \pi(b)\xi_n\rVert},
$$ 
are unit vectors for $n \geq N$. Our goal now is to show that 
$$
\lim_{n \to \infty} \lVert \pi(b)\eta_n - q^{-2l} \lambda \eta_n \rVert = 0,
$$ 
which is the same as asserting $q^{-2l} \lambda \in \spec(\pi(a))$. Assuming $n \geq N$, hence $\lVert \pi(b)\xi_n \rVert \geq C$, and using the relation $ab=q^{-2l}ba$ we obtain 
\begin{equation}
\begin{split}
\lVert \pi(a)\eta_n - q^{-2l}\lambda \eta_n \rVert \notag
&= \frac{\lVert \pi(a) \pi(b) \xi_n-q^{-2l} \lambda \pi(b)\xi_n \rVert}{\lVert \pi(b)\xi_n \rVert} \notag\\ &\leq \frac{\lVert \pi(b) \rVert}{q^{2l}C} \lVert \pi(a)\xi_n - \lambda \xi_n \rVert \xymatrix{ \ar[rr]_-{n\to \infty} && 0.} \notag 
\end{split}
\end{equation}

Hence we have shown that if $\lambda \in \spec(\pi(a))$ then $ q^{-2l}\lambda \in \spec(\pi(a))$. So the spectrum contains $\lambda, q^{-2l}\lambda, q^{-4l} \lambda, ... .$ Since we require this sequence to be bounded there must exist an $n \in \NN $ such that $q^{-2nl} \lambda = \lambda_0$, the largest possible eigenvalue, i.e.\ $\lambda=q^{2nl}\lambda_0$ for some $n \in \NN $. Hence we have shown that $\spec(\pi(a)) \subset \{q^{2nl}\lambda_0 :n \in \NN \}\cup \{0\}.$
The implication of this calculation is that there exists a unit vector $\xi$ such that $\pi(a)\xi =\lambda_0 \xi$. We now use this fact to  calculate directly the representations.

It follows from the relation $ab^{*^p} = q^{2lp} b^{*^p} a$ that 
$$
\pi(a)(\pi(b^{*^p})\xi )=q^{2lp}\lambda_0\pi(b^{*^p})\xi ,
$$ 
which shows that $\pi(b^{*^p})\xi $ are eigenvectors of $\pi(a)$ which have distinct eigenvalues $q^{2lp}\lambda_0$. Thus the vectors:
$$
e_p=\frac{\pi(b^{*^p})\xi }{\lVert \pi(b^{*^p})\xi  \rVert}, \quad p \in \NN,
$$ 
form an orthonormal system. We now show that the span of the $e_p$ is closed under the action of the algebra:
$$
\pi(a)e_p=\frac{\pi(a) \pi(b^{*^p})\xi }{\lVert \pi({b^{*^p}})\xi  \rVert}=\frac{\pi(q^{2lp} b^{*^p}a)\xi }{\lVert \pi({b^{*^p}})\xi \rVert}=\frac{q^{2lp}\lambda_0 \pi(b^{*^p})\xi }{\lVert \pi(b^{*^p}) \rVert},
$$
therefore,
$$
\pi(a)e_p=q^{2lp}\lambda_0 e_p.
$$ 
On the other hand,
$$
\pi(b)e_p=\frac{\pi(b)\pi(b^{*^p})\xi }{\Vert \pi(b^{*^p})\xi \rVert}=\frac{\pi(bb^*)\pi(b^{*^{p-1}})\xi }{\lVert \pi(b^{*^p})\xi  \rVert}.
$$
Now,
\begin{equation}
\begin{split}
\lVert \pi(b^{*^p})\xi \rVert& = \langle \pi(b^{*^p})\xi ,\pi(b^{*^p})\xi \rangle^{\frac{1}{2}}\notag\\&=\langle \pi(b) \pi(b^{*^p})\xi ,\pi(b^{*^{p-1}})\xi  \rangle^{\frac{1}{2}}\notag\\&=\langle q^{2klp}\lambda_0^{2k} \prod_{r=1}^{l} \mathcal (1- \lambda_0 q^{2(l(p-1)+r)})\pi(b^{*^{p-1}})\xi ,\pi(b^{*^{p-1}})\xi  \rangle^{\frac{1}{2}}\notag\\&=q^{klp}\lambda_0^k \prod_{r=1}^{l} \mathcal(1-\lambda_0 q^{2(l(p-1)+r)})^{\frac{1}{2}} \lVert \pi(b^{*^{p-1}})\xi  \lVert.\notag
\end{split}
\end{equation}
Hence
$$
\pi(b)e_p=\frac{\pi(b) \pi(b^{*^p})\xi }{\lVert \pi(b^{*^p})\xi  \rVert}=\frac{q^{2lpk}\lambda_0^{2k} \prod_{r=1}^{l} \mathcal (1-\lambda_0 q^{2(l(p-1)+r)}) \pi(b^{*^{p-1}})}{q^{lpk}\lambda_0^k \prod_{r=1}^{l} \mathcal (1-\lambda_0 q^{2(l(p-1)+r)})^{\frac{1}{2}} \lVert \pi(b^{*^{p-1}})\rVert},
$$ 
which simplifies to
$$
\pi(b)e_p=q^{lpk}\lambda_0^k \prod_{r=1}^{l} \mathcal (1-\lambda_0 q^{2(l(p-1)+r)})^{\frac{1}{2}} e_{p-1}.
$$ 
Reversing the order of multiplication by using the substition $r'=l-r$ we arrive at the following result:
$$
\pi(b)(e_p)=q^{lpk}\lambda_0^k \prod_{r'=1}^{l} \mathcal (1-\lambda_0 q^{2(lp- r')})^{\frac{1}{2}} e_{p-1}.
$$
Considering the case when $p=0$ we see that $\pi(b)e_0=\pi(b)\xi =0$ since $\pi(b)$ acts as a stepping down operator on the basis elements $e_p$. This implies that
$$
\lambda_0^k \prod_{r'=1}^{l} (1-\lambda_0 q^{-2r'})=0,
$$
therefore $\lambda_0=0$, which corresponds to the one dimensional case, or $\lambda_0=q^{2s}$ for some $s \in \{1,2,...,l\}$, which relates to the infinite dimensional case. 
\end{proof}

\begin{proposition}\label{prop.faith}
Each of the representations $\pi_s$ is a faithful representation.
\end{proposition}
\begin{proof}
We use reasoning similar to that in the proof of \cite[Proposition~1]{HajMat:gra}. Consider an arbitrary element of $\cO(\WP_q(k,l))$ expressed as a linear combination of the basis \eqref{basis},
$$
x = \sum_{m,n} \mu_{m,n} a^mb^n + \sum_{m,n'} \nu_{m,n'}a^mb^{*n'},
$$
$\mu_{m,n},\nu_{m,n'}\in \CC$, and suppose that $\pi_s(x)=0$, i.e.\ $\pi_s(x) e^s_p =0$, for all $p\in \NN$. Since the application of $\pi_s(a^mb^n)$ to $e_p^s$ does not increase the index $p$, while the application of $\pi_s(a^mb^{*n'})$ increases $p$, the vanishing condition splits into two cases, which can be dealt with separately, and we deal only with the first of them. The condition 
$$
\pi_s\left(\sum_{m,n} \mu_{m,n} a^mb^n \right) e^s_p =0, \qquad \mbox{for all}\ p\in \NN,
$$
is equivalent to the system of equations
$$
\sum_{m,n} \mu_{m,n}q^{nk\left[l\left(p -(n-1)/2\right)+s \right]}q^{2m[l(p-n)+s]}\prod_{r=1}^{ln}\left(1 - q^{2(lp+s-r)}\right)^{1/2} e^s_{p-n} = 0, \quad \mbox{for all}\ p\in \NN ,
$$
(use \eqref{wpkl} repetitively). Since this must be true for all $n\leq p$, and the vectors $e^s_{p-n}$ are linearly independent for different $n$, we obtain a system of equations, one for each $n$,
\begin{equation}\label{lin}
\sum_{m} \mu_{m,n}q^{2m[l(p-n)+s]}=0, \qquad \mbox{for all}\ p\in \NN.
\end{equation}
There are only finitely many non-zero  coefficients $\mu_{m,n}$. Let $N$ be the smallest natural number such that $\mu_{m,n}= 0$, for all $m>N$. Define 
$$
\lambda_{p,n} = q^{2[l(p-n)+s]}.
$$
Then equations \eqref{lin} for $\mu_{m,n}$ take the form
$$
\sum_{m=0}^N \mu_{m,n}\lambda_{p,n}^m=0, \qquad \mbox{for all}\ p\in \NN.
$$
The matrix of coeffcients of the first $N+1$ equations (for $p=0,1,\ldots , N$) has the Vandermonde form, and is invertible since  $\lambda_{p,n}\neq \lambda_{p',n}$ if $p\neq p'$ (remember that $q\in( 0,1)$). Therefore, $\mu_{m,n} = 0$ is the only solution to  \eqref{lin}. Similarly one proves that necessarily $\nu_{m,n'} =0$ and concludes that $\pi_s$ is an injective map.
\end{proof}

Finally, we look at the way representations of $\cO(\WP_q(k,l))$ are related to representations of the quantum sphere algebra $\cO(S^{3}_q)$.

\begin{proposition}\label{lemma.direct}
Let $\pi : \cO(S^{3}_q) \to \rend{} V$ denote the  representation  of $\cO(S^{3}_q)$ given on an orthonormal basis $e_n$, $n\in \NN$ of $V$ by \cite{VakSob:su2}
\begin{equation}\label{su2rep}
{\pi}(\alpha)e_n = \left(1-q^{2n}\right)^{1/2}e_{n-1}, \qquad {\pi}(\beta)e_n = q^{n+1}e_n, 
\end{equation}
(see  also \cite{Wor:twi}). 
Then there exists an algebra isomorphism $\phi :  \rend{} {\bigoplus_{s=1}^{l}V_s} \to \rend{} V$ rendering commutative the following diagram, 
\begin{equation}\label{diag}
\xymatrix {{\cO(\WP_q(k,l))} \ar[rr]^-\theta \ar[d]_{\oplus_s \pi_s} && \cO(S^{3}_q) \ar[d]^\pi \\
 \rend{} {\bigoplus_{s=1}^{l}V_s} \ar[rr]^\phi && \rend{} V,}
\end{equation}
where $\theta$ is the $*$-algebra map given by formulae \eqref{theta}. Consequently, $\theta$ is a monomorphism.
\end{proposition}
\begin{proof}
Consider the vector space isomorphism,
$$
\hat{\phi} : \bigoplus_{s=1}^{l}V_s \to V, \qquad e_p^s\mapsto e_{lp+s-1},
$$
and let $\phi: \rend{} {\bigoplus_{s=1}^{l}V_s}\to \rend{} V$, $f\mapsto \hat{\phi}\circ f\circ \hat{\phi}^{-1}$,  be the induced algebra isomorphism. Using \eqref{su2rep} one easily finds that
$$
\pi(\alpha^l\beta^k) e_n = q^{k(n+1)}\prod_{r=1}^{l}\left(1 - q^{2(n-r+1)}\right)^{1/2}e_{n-1}, \qquad \pi(\beta\beta^*) e_n = q^{2(n+1)}e_n.
$$
This immediately implies that, for all $x\in \cO(\WP_q(k,l))$,
$$
\hat{\phi}\left(\pi_s(x)e_p^s\right) = \pi\left(\theta(x)\right) \hat{\phi}(e_p^s).
$$
Therefore, the induced map $\phi$ makes the diagram \eqref{diag} commute as required.

By Proposition~\ref{prop.faith}, the representations $\pi_s$ are faithful, so the map $\oplus_s \pi_s: \cO(\WP_q(k,l))\to \rend{} {\bigoplus_{s=1}^{l}V_s}$ is injective. The commutativity of \eqref{diag} implies that $\theta$ is an injective map.
\end{proof}

\section{Quantum teardrops and quantum principal bundles}
\setcounter{equation}{0}
\begin{definition}\label{def.princ}
Let $H$ be a Hopf algebra with bijective antipode and let $A$ be a right $H$-comodule algebra with coaction $\varrho: A\to A\ot H$. Let $B$ denote the coinvariant subalgebra of $A$, i.e.\ $B=A^{coH}:= \{b\in A\; |\; \varrho(b) = b\ot 1\}$. We say that $A$ is a {\em principal $H$-comodule algebra} if:
\begin{blist}
\item the coaction $\varrho$ is free, that is the canonical map
$$
\can: A\ot_B A\to A\ot H, \qquad a\ot a'\mapsto a\varrho(a'),
$$
is bijective (the Hopf-Galois condition);
\item there exists a strong connection in $A$, that is  the multiplication map
$$
B\ot A \to A, \qquad b\ot a\mapsto ba,
$$
splits as a left $B$-module and right $H$-comodule map (the equivariant projectivity).
\end{blist}
\end{definition}
As indicated in \cite{Sch:pri}, \cite{BrzMaj:gau} or \cite{Haj:str} and argued more recently in e.g.\ \cite{Brz:syn} or \cite{HajKra:pie}, principal comodule algebras should be understood as principal bundles in noncommutative geometry. In the first part of this section we prove that only in the case of  the quantum 2-sphere (i.e.\ $k=l=1$), $\cO(S_q^3)$ is a principal $\cO(U(1))$-comodule algebra over  $\cO(\WP_q(k,l))$. In that case we are dealing with the well-known quantum version of the Hopf fibration.

\begin{theorem}\label{H-G}
The algebra of $\cO(S_q^3)$ is a principal $\cO(U(1))$-comodule algebra over $\cO(\WP_q(k,l))$ by the coaction $\varrho_{k,l}$ if and only if $k=l=1$.
\end{theorem}
\begin{proof}
If $k=l=1$, then $\cO(\WP_q(1,1)) = \cO(S_q^2)$, and it is known that $\cO(S_q^3)$ is a principal comodule algebra that describes the quantum Hopf fibration (over the standard Podle\'s sphere); see \cite{BrzMaj:gau} or \cite{HajMaj:pro}. We assume, therefore, that $k\neq l$ (i.e.\ $(k,l)\neq (1,1)$), and show that $1\otimes u$ is not in the image of the canonical map in that case. 
We proceed by identifying a basis for $\cO(SU_q(2)) \otimes \cO(SU_q(2))$ and applying the canonical map to observe the form of the codomain.

A basis for $\cO(SU_q(2)) \otimes \cO(SU_q(2))$ consists of
\begin{equation}
\begin{split}
\alpha^h \beta^m {\beta^*}^n \otimes \alpha^{\bar{h}} \beta^{\bar{m}} {\beta^*}^{\bar{n}}, \qquad & \ \ \notag 
\alpha^h \beta^m {\beta^*}^n \otimes \beta^{\bar{m}} {\beta^*}^{\bar{n}} {\alpha^*}^{\bar{p}},\\
\beta^m {\beta^*}^n {\alpha^*}^p \otimes \alpha^{\bar{h}} \beta^{\bar{m}} {\beta^*}^{\bar{n}}, \qquad & \ \ \notag
\beta^m {\beta^*}^n {\alpha^*}^p \otimes \beta^{\bar{m}} {\beta^*}^{\bar{n}} {\alpha^*}^{\bar{p}},
\end{split}
\end{equation}
where $h,m,n,\bar{h}, \bar{m}, \bar{n}, \bar{p} \in \NN$. 
Hence, applying the canonical map we conclude that every element in the image of $\can$ is a linear combination of:
\begin{equation}\label{im.can}
\begin{split}
\alpha^{h+\bar{h}} \beta^{m+\bar{m}} {\beta^*}^{n+\bar{n}} \otimes u^{k \bar{h}-l \bar{m}+l\bar{n}} \ , \quad & 
\beta^{m+\bar{m}} {\beta^*}^{n+\bar{n}} \alpha^h {\alpha^*}^{\bar{p}} \otimes u^{-l\bar{m}+l\bar{n}-k\bar{p}} \ , \ \\
\beta^{m+\bar{m}} {\beta^*}^{n+\bar{n}} {\alpha^*}^p \alpha^{\bar{h}} \otimes u^{k\bar{h}-l\bar{m}+l\bar{n}} \ , \quad & 
\beta^{m+\bar{m}} {\beta^*}^{n+\bar{n}} {\alpha^*}^{p+\bar{p}} \otimes u^{-l\bar{m}+l\bar{n}-k\bar{p}} \ ,
\end{split}
\end{equation}
where 
$
h,m,n,\bar{h}, \bar{m}, \bar{n}, \bar{p} \in \NN.
$

To obtain identity in the first leg we must use one of the following relations \eqref{su.2} or equations which include terms of the form ${\alpha^*}^n \alpha^m$ or $\alpha^n {\alpha^*}^m$. A straightforward calculation gives the following: 
\begin{equation}
\alpha^m {\alpha^*}^n=
\begin{cases}
\prod_{p=1}^{m} \mathcal(1-q^{2p-2}\beta {\beta^*}) & \mbox{when \ $m=n$,}\\
\alpha^{m-n} \prod_{p=1}^{m} \mathcal(1-q^{2p-2}\beta {\beta^*}) & \mbox{when \ $m > n$,}\\
\prod_{p=1}^{m} \mathcal(1-q^{2p-2}\beta {\beta^*}) {\alpha^*}^{n-m}  & \mbox{when\ $n > m$,}
\end{cases}\notag
\end{equation}
and
\begin{equation}
{\alpha^*}^n \alpha^m=
\begin{cases}
\prod_{p=1}^{m} \mathcal(1-q^{-2p}\beta {\beta^*}) &\mbox{when\ $m=n$,}\\
{\alpha^*}^{m} \prod_{p=1}^{m} \mathcal(1-q^{-2p}\beta {\beta^*}) & \mbox{when \ $n > m$,}\\
\prod_{p=1}^{m} \mathcal(1-q^{-2p}\beta {\beta^*}) \alpha^{m-n}  & \mbox{when \ $n > m$.}
\end{cases}\notag
\end{equation}

We see that to obtain identity in the first leg we require the powers of $\alpha$ and ${\alpha^*}$ to be equal. We now construct all possible elements of the domain which map to $1 \otimes u$ after applying the canonical map.

\underline{Case 1}: Use the second term in \eqref{im.can} to obtain $\alpha^N {\alpha^*}^N$. In this case: $h=\bar{p}=N$, $n+\bar{n}=m+\bar{m}=0$. Since $n, \bar{n},m,\bar{m} \in \NN$ we must have $n=\bar{n}=m=\bar{m}=0$. Also $-l\bar{m}+l\bar{n}-k\bar{p}=1$, which implies that $-k \bar{p}=1$, hence there are no possible terms. 
 
\underline{Case 2}: Use the third term in \eqref{im.can} to obtain ${\alpha^*}^N \alpha^N$. In this case $\bar{h}=p=N$, $m=\bar{m}=n=\bar{n}=0$. Also $k \bar{h}-l\bar{m}+l\bar{n}=1$, which implies that  $k \bar{h}=1$, hence $k=1$ and $\bar{h}=1$. Therefore, the only terms of the form ${\alpha^*}^N \alpha^N$ are when $N=1$ and in this case $k=1$. We now look at the other terms which are of the form $\beta {\beta^*}$ so that we can use the relation ${\alpha^*} \alpha + q^{-2} \beta {\beta^*} =1$. Four possibilities need be considered, one for each of the terms in \eqref{im.can}. 
In the case of the first of these terms $h=\bar{h}=0$, $m+\bar{m}=1$, $n+\bar{n}=1$ and  $k \bar{h}-l\bar{m}+l\bar{n}=1$, which implies that  $ l(\bar{n}-\bar{m})=1$, hence $ l=1$ and $\bar{n}-\bar{m}=1.$ The only solution is: $l=1$, $n=\bar{m}=0$, $m=\bar{n}=1$, $h=\bar{h}=0.$ A similar approach can be used when considering the remaining terms in \eqref{im.can} to conclude that in all four cases one is forced to require $l=1$. Therefore, it is impossible to obtain a term of the form $1 \otimes u$ when both $k$ and $l$ are not simultaneously equal to one. This implies that $\cO(\WW \PP_q(k,l)) \subseteq \cO(SU_q(2))$ is not a Hopf-Galois extension when $k$ and $l$ are not both one.
\end{proof}

Theorem~\ref{H-G} asserts that the defining action \eqref{rhokl} of $U(1)$ on the quantum group $SU_q(2)$ does not make it a total space of a quantum $U(1)$-principal bundle over the quantum weighted projective space $\WW\PP_q(k,l)$, unless $k=l=1$ (the case of the quantum Hopf fibration). The remainder of this section  is devoted to construction of a quantum $U(1)$-principal bundle over the quantum teardrop $\WW\PP_q(1,l)$ with the total space provided by the {\em quantum lens space} $L_q(l;1,l)$. 

The coordinate algebra of the {\em quantum lens space} $\cO(L_q(l;1,l))$ is defined as follows \cite{HonSzy:len}. The coordinate (or group) algebra of the cyclic group $\ZZ_l$, $\cO(\ZZ_l)$, is a Hopf $*$-algebra generated by a unitary grouplike element $w$ satisfying  $w^l=1$. $\cO(SU_q(2))$ is a right $\cO(\ZZ_l)$-comodule $*$-algebra with the coaction 
$$
\varrho: \cO(SU_q(2))\to \cO(SU_q(2))\ot \cO(\ZZ_l), \qquad \alpha\mapsto \alpha\ot w, \quad \beta\mapsto \beta\ot 1.
$$
$\cO(L_q(l;1,l))$ is defined as the coinvariant subalgebra of $\cO(SU_q(2))$ under the coaction $\varrho$. Following similar arguments to those in Section~\ref{sec.alg} one easily finds that $\cO(L_q(l;1,l))$ is a $*$-subalgebra of $\cO(SU_q(2))$ generated by  $c := \alpha^l$ and $d:=\beta$. These elements satisfy the following relations:
\begin{subequations} \label{rel.lens}
\begin{gather}
cd = q^{l}dc, \qquad cd^* = q^{l}d^*c, \qquad dd^* = d^*d, \label{rel.lens1}\\
cc^* = \prod_{m=0}^{l-1}(1-q^{2m}dd^*), \qquad c^*c = \prod_{m=1}^l(1-q^{-2m}dd^*). \label{rel.lens2}
\end{gather}
\end{subequations}
Universally, $\cO(L_q(l;1,l))$ can be defined as a $*$-algebra generated by $c$ and $d$ subject to relations \eqref{rel.lens}. 

Again, following the same techniques as in Section~\ref{sec.alg} one can classify -- up to unitary equivalence -- all irreducible $*$-representations of $\cO(L_q(l;1,l))$. There is a family of one-dimensional representations $\pi_0^\lambda$ defined as
$$
\pi_0^\lambda(c) = \lambda, \qquad \pi_0^\lambda(d) = 0, \qquad  \lambda \in \CC,\; |\lambda|=1.
$$
For every $s=1,2,\ldots , l$, there is a family of infinite-dimensional representations $\pi_s^\lambda$, $\lambda\in \CC$, $|\lambda|=1$. The action of $\pi_s^\lambda$ on an orthonormal basis $e_p^{\lambda, s}$, $p\in \NN$,  for its representation space $V^\lambda_s\simeq l^2(\NN)$ is given by
$$
\pi_s^\lambda(c) e_p^{\lambda, s} =\prod_{m=1}^l(1-q^{2(pl+s-m)})^{1/2}e_{p-1}^{\lambda, s}, \qquad 
 \pi_s^\lambda(d) e_p^{\lambda, s} = \lambda q^{pl+s}e_p^{\lambda, s}.
$$
As for quantum teardrops, there is a vector space isomorphism
$$
{\phi}^\lambda : \bigoplus_{s=1}^{l}V_s^\lambda \to V^\lambda, \qquad e_p^{\lambda,s}\mapsto e_{lp+s-1}^\lambda,
$$
which embeds the direct sum of representations $\pi_s^\lambda$ in the representation $\pi^\lambda$ of $\cO(SU_q(2))$,
$$
{\pi}^\lambda(\alpha)e_n^\lambda = \left(1-q^{2n}\right)^{1/2}e_{n-1}^\lambda, \qquad {\pi}^\lambda(\beta)e_n^\lambda = \lambda q^{n+1}e_n^\lambda.
$$
Here $e_n^\lambda$, $n\in \NN$, is an orthonormal basis for the representation space $V^\lambda$.

$\cO(L_q(l;1,l))$ is a right comodule algebra over the Hopf algebra $\cO(U(1)) = \CC[u,u^*]$, $u^{-1} = u^*$. The coaction $\varrho_l: \cO(L_q(l;1,l))\to \cO(L_q(l;1,l))\ot \cO(U(1))$ is given on generators $c$ and $d$ by
$$
\varrho_l: c\mapsto c\ot u , \qquad d\mapsto d\ot u^*.
$$
It is an easy exercise to check that 
$$
\cO(L_q(l;1,l))^{co \cO(U(1))} \simeq \cO(\WW\PP_q(1,l)),
$$ 
through the identification $a = cd$, $b= dd^*$.

\begin{theorem}\label{thm.H-G}
The coordinate algebra of the quantum lens space $\cO(L_q(l;1,l))$ is a principal $\cO(U(1))$-comodule algebra over $\cO(\WW\PP_q(1,l))$.
\end{theorem}
\begin{proof}
As explained in \cite{DabGro:str}, \cite{BrzHaj:Che}, a right $H$-comodule algebra $A$ with coaction $\varrho: A\to A\ot H$  is principal if and only if it admits a {\em strong connection}, that is if there exists a map
$
\omega:H\longrightarrow A\otimes A,
$
such that
\begin{subequations}
\label{strong}
\begin{gather}
\omega(1)=1\otimes 1, \label{strong1}\\
\mu\circ \omega = \eta \circ \varepsilon, \label{strong2}\\
 (\omega\otimes\id)\circ\Delta  = (\id\otimes \varrho)\circ \omega , \label{strong3}\\
(S\otimes \omega)\circ\Delta = (\sigma\otimes \id)\circ (\varrho\otimes \id)\circ \omega .  \label{strong4}
\end{gather}
\end{subequations}
Here $\mu: A\ot A\to A$ denotes the multiplication map, $\eta: \CC\to A$ is the unit map, $\Delta: H\to H\ot H$ is the comultiplication, $\eps: H\to\CC$ counit and $S: H\to H$ the (bijective) antipode of the Hopf algebra $H$, and $\sigma : A\ot H \to H\ot A$ is the flip.

A strong connection for $\cO(L_q(l;1,l))$ is defined recursively as follows. Set $\omega(1) = 1\ot 1$ and then, for all $n>0$,
\begin{subequations}
\begin{gather}
\omega(u^n) = c^*\omega(u^{n-1})c -\sum_{m=1}^l (-1)^mq^{-m(m+1)} \binom  l m_{\! q^{-2}}d^m {d^*}^{m-1}\omega(u^{n-1})d^*, \label{eq.strong1}\\
\omega(u^{-n}) = c\omega(u^{-n+1})c^* -\sum_{m=1}^l (-1)^mq^{m(m-1)} \binom  l m_{\! q^{2}}d^{m -1}{d^*}^{m}\omega(u^{-n+1})d,  \label{eq.strong2}
\end{gather}
\end{subequations}
where, for all $x\in \RR$, the {\em deformed} or {\em q-binomial} coefficients $\binom  l m_{\! x}$  are defined by the following polynomial equality in indeterminate $t$
\begin{equation}\label{binom}
\prod_{m=1}^l (1+ x^{m-1}t) = \sum_{m=0}^l x^{m(m-1)/2}\binom  l m_{\! x} t^m.
\end{equation}

By construction, $\omega$ has property \eqref{strong1}. The remaining properties are proven by induction on $n$. That $\mu(\omega(u)) = \eps(u)=1$ follows by the second of equations \eqref{rel.lens2} combined with \eqref{binom}. Applying $\id\ot \varrho_l$ to the right hand side of  \eqref{eq.strong1} (with $n=1$), one immediately obtains that $(\id\ot \varrho_l)(\omega(u)) = \omega(u)\ot u$, as required for \eqref{strong3}. Similarly one checks \eqref{strong4}. 

Now, assume that $\omega(u^{n-1})$ satisfies conditions \eqref{strong2}--\eqref{strong4}. Then, multiplying the right hand side of  \eqref{eq.strong1} and using $\mu(\omega(u^{n-1})) = \eps(u^{n-1}) = 1$, we obtain
$$
\mu(\omega(u)) = c^*c -\sum_{m=1}^l (-1)^mq^{-m(m+1)} \binom  l m_{\! q^{-2}}d^m {d^*}^{m} =1,
$$
by \eqref{rel.lens2} and \eqref{binom}. Since $\varrho_l$ is an algebra map and   $(\id\ot \varrho_l)(\omega(u^{n-1})) = \omega(u^{n-1})\ot u^{n-1}$ by inductive assumption, we can compute
\begin{equation}
\begin{split}
(\id\ot   \varrho_l) & (\omega(u^{n}))  =  c^*(\id\ot \varrho_l)(\omega(u^{n-1})c)\notag\\
& -  \sum_{m=1}^l (-1)^mq^{-m(m+1)} \binom  l m_{\! q^{-2}}d^m {d^*}^{m-1}(\id\ot \varrho_l)(\omega(u^{n-1})d^*)\notag\\
& =  c^*\omega(u^{n-1})c\ot u^n 
 -\sum_{m=1}^l (-1)^mq^{-m(m+1)} \binom  l m_{\! q^{-2}}d^m {d^*}^{m-1}\omega(u^{n-1})d^*\ot u^n\notag \\
&  =  \omega(u^n)\ot u^n, \notag
\end{split}
\end{equation}
as required. The left colinearity of $\omega$ \eqref{strong4} is proven in a similar way. The case of $\omega(u^{-n})$ is treated in the same manner.
\end{proof}

Recall that a principal $H$-comodule algebra $A$ is said to be {\em cleft} if there exists a convolution invertible, right $H$-colinear map $j: H\to A$ such that $j(1)=1$. The special case of this is when $j$ is an algebra map (then its convolution inverse is $j\circ S$) and this corresponds to the trivial quantum principal bundle. In the case of comodule algebras over $\cO(U(1)) = \CC[u,u^{-1}]$ the necessary condition for a map $j: \cO(U(1)) \to A$ to be convolution invertible is that $j(u)$ is an invertible element (unit) of $A$. Arguing as in \cite[Appendix]{HajMaj:pro} we obtain
\begin{lemma}
The principal $\cO(U(1))$-comodule algebra  $\cO(L_q(l;1,l))$ is not cleft.
\end{lemma}
\begin{proof}
Multiples of $1$ are the only invertible elements of $\cO(SU_q(2))$; see \cite[Appendix]{HajMaj:pro}. Since $\cO(L_q(l;1,l))$ is a subalgebra of $\cO(SU_q(2))$ the same can be said about $\cO(L_q(l;1,l))$. Thus any convolution invertible map $j: \cO(U(1))\to \cO(L_q(l;1,l))$ must have the form $j(u) = \lambda 1$, for some $\lambda\in \CC^*$. This, however, violates the right $\cO(U(1))$-colinearity of $j$.
\end{proof}

The surjectivity of the canonical map in Definition~\ref{def.princ} corresponds to the freeness of the coaction $\varrho$ of $H$ on $A$. By Theorem~\ref{H-G} we know that if $(k,l)\neq (1,1)$, then the coaction $\varrho_{k,l}$ of $\cO(U(1))$ on $\cO(SU_q(2))$ is not free. However, Theorem~\ref{thm.H-G} implies that $\varrho_{1,l}$ is {\em almost free} in the following sense.

\begin{definition}\label{def.almost.free}
Let $H$ be a Hopf algebra and let $A$ be a right $H$-comodule algebra with coaction $\varrho: A\to A\ot H$. We say that the coaction is {\em almost free} if the cokernel of  the (lifted) canonical map
$$
\overline{\can}: A\ot  A\to A\ot H, \qquad a\ot a'\mapsto a\varrho(a'),
$$
is finitely generated as a left $A$-module.
\end{definition}

\begin{corollary}\label{cor.almost}
The coaction $\varrho_{1,l}$ is almost free.
\end{corollary}
\begin{proof}
Note that the $*$-algebra inclusion 
$$
\iota: \cO(L_q(l;1,l)) \hookrightarrow \cO(SU_q(2)), \qquad c\mapsto \alpha^l, \quad d\mapsto \beta ,
$$
makes the following diagram commute
$$
\xymatrix{\cO(L_q(l;1,l)) \ar[rr]^{\iota}\ar[d]_{\varrho_{l}} && \cO(SU_q(2))\ar[d]^{\varrho_{1,l}} \\
\cO(L_q(l;1,l))\ot\cO(U(1)) \ar[rr]^-{\iota\ot(-)^l} && \cO(SU_q(2))\ot\cO(U(1)),}
$$
where $(-)^l: u\to u^l$. The surjectivity of the canonical map $\cO(L_q(l;1,l))\ot\cO(L_q(l;1,l)) \to \cO(L_q(l;1,l))\ot\cO(U(1))$ (proven in Theorem~\ref{thm.H-G}) implies that $1\ot u^{ml} \in \im (\overline{\can})$, $m\in \ZZ$, where $\overline{\can}$ is the (lifted) canonical map corresponding to coaction $\varrho_{1,l}$. This means that $\cO(SU_q(2))\ot\CC[u^l,u^{-l}] \subseteq \im (\overline{\can})$. Therefore, there is a short exact sequence of left $\cO(SU_q(2))$-modules
$$
\xymatrix{(\cO(SU_q(2))\ot\CC[u,u^{-1}])/(\cO(SU_q(2))\ot\CC[u^l,u^{-l}])\ar[r] & \coker (\overline{\can}) \ar[r] & 0.}
$$
The left $\cO(SU_q(2))$-module $ (\cO(SU_q(2))\ot\CC[u,u^{-1}])/(\cO(SU_q(2))\ot\CC[u^l,u^{-l}])$ is finitely generated, hence so is $\coker (\overline{\can})$.
\end{proof}

\section{Fredholm modules and the Chern-Connes pairing for quantum teardrops}\label{sec.Fred}
\setcounter{equation}{0}
In this section first we follow \cite{MasNak:dif} (see also \cite{HawLan:fred} and \cite{DAnLan:bou}), and associate even Fredholm modules to algebras $\cO(\WW\PP_q(k,l))$ and use them to construct traces or cyclic cocycles on $\cO(\WW\PP_q(k,l))$. The latter are  then used  to calculate the Chern number of a non-commutative line bundle associated to the quantum principal bundle $\cO(L_q(l;1,l))$ over the quantum teardrop $\cO(\WW\PP_q(1,l))$.

Recall from \cite[Chapter~IV]{Con:non} that an {\em even Fredholm module} over a $*$-algebra $A$ is a quadruple $(\cV,\pi, F,\gamma)$, where $\cV$ is a Hilbert space of a representation $\pi$ of $A$ and $F$ and $\gamma$ are operators on $\cV$ such that, $F^* =F$, $F^2 =I$ and for all $a\in A$, the commutator $[F,\pi(a)]$ is a compact operator, and $\gamma^2 =I$, $\gamma F = -\gamma F$. A Fredholm module is said to be {\em 1-summable} if $[F,\pi(a)]$ is a trace class operator for all $a\in A$. In this case one associates to $(\cV,\pi, F,\gamma)$  a cyclic cocycle  on $A$ or a {\em Chern character}  $\tau$ by
$$
\tau (a) = \tr (\gamma \pi(a)).
$$

\begin{proposition}\label{prop.Fred}
For every $s=1,2,\ldots ,l$, let $\cV_s = V_s\oplus V_0$, where $V_s$ is the Hilbert space $l^2(\NN)$ of representation $\pi_s$ and $V_0 = \oplus \CC \simeq  l^2(\NN)$, which we take to be representation space of $\pi= \oplus \pi_0$; see Proposition~\ref{prop.rep}. Define $\overline{\pi}_s := \pi_s \oplus \pi$, 
$$
 F = \begin{pmatrix} 0 & I \cr I & 0\end{pmatrix}, \qquad \gamma = \begin{pmatrix} I & 0 \cr 0 & -I\end{pmatrix}.
 $$
 Then $(\cV_s, \overline{\pi}_s, F,\gamma)$ are  1-summable Fredholm modules over $\cO(\WW\PP_q(k,l))$. The corresponding Chern characters are
 \begin{equation}\label{tau}
 \tau_s (a^mb^n) =
 \begin{cases} \frac{q^{2ms}}{1-q^{2ml}}   & \mbox{if $n=0$, $m\neq 0$},\\
 0  & \mbox{otherwise}.
 \end{cases}
\end{equation}
 Here $n\in \ZZ$ and, for a positive $n$, $b^{-n}$ means ${b^*}^n$.
\end{proposition}
\begin{proof}
It is obvious that $F^* =F$, $F^2 = \gamma^2 =I$ and $F\gamma + \gamma F =0$. Next, by a straightforward calculation, for all $x\in \cO(\WW\PP_q(k,l))$,
$$
[F,\pi(x)] = \begin{pmatrix} 0 & \pi(x) - \pi_s(x) \cr \pi_{s}(x) - \pi(x) & 0\end{pmatrix}.
 $$
Using the formulae in Proposition~\ref{prop.rep} one easily finds, for all $m,n, p\in \NN$,
$$
\pi_s(a^mb^n) e_p^s = q^{nk[lp-(n-1)/2 +s] + 2m[l(p-n)+s]}\prod_{r=1}^{ln}(1-q^{2(lp+s-r)})^{1/2} e^s_{p-n};
$$
compare the proof of Proposition~\ref{prop.faith}. This implies that, for positive $m$,   $\pi_s(a^mb^n)$ are trace class operators, as $\tr (\pi_s(a^mb^n)) =0$ if $n\neq 0$ and 
\begin{equation}\label{trace}
\tr (\pi_s(a^m)) = \sum_p q^{2m(lp+s)} = \frac{q^{2ms}}{1-q^{2ml}} .
\end{equation}
Since $\pi_0(a^mb^n) =0$, if $(m,n)\neq (0,0)$, and $\pi_0(1) =1$, we conclude that, for all $x\in \cO(\WW\PP_q(k,l))$, $\pi(x)-\pi_s(x)$ is a trace class operator. Therefore, $(\cV_s, \overline{\pi}_s, F,\gamma)$  is a 1-summable Fredholm module. 

Finally, 
$$
\tau_s(x) = \tr (\gamma \overline{\pi}_s(x)) = \tr (\pi_s(x) - \pi(x)),
$$
and the formula \eqref{tau} follows by equation \eqref{trace}.
\end{proof}

Note that the form of the Chern character of $(\cV_s, \overline{\pi}_s, F,\gamma)$  is independent of $k$. In the case $l=1$, necessarily $s=1$ and $\tau_1$ coincides with the trace calculated for the quantum 2-sphere in \cite{MasNak:dif}. Similarly to the case studied in \cite{MasNak:dif}, the characters $\tau_s$ on $\cO(\WW\PP_q(k,l))$ factor through the algebra map 
$$
\cO(\WW\PP_q(k,l)) \to \CC[a], \qquad a^mb^n \mapsto \delta_{n,0} a^m.
$$
On the polynomial algebra $\CC[a]$ the characters are given by Jackson's integrals. More precisely, define $\hat{\tau}_s$ by the commutative diagram
$$
\xymatrix{\cO(\WW\PP_q(k,l))\ar[dr] \ar[rr]^-{\tau_s} &&\CC \\
& \CC[a] \ar[ur] _{\hat{\tau}_s}. &}
$$
Then
$$
\hat{\tau}_s (f) = \frac{1}{1-q^{2l}} \int_0^{q^{2s}} \frac{f(a)}{a} d_{q^{2l}} a,
$$
where the Jackson integral is defined by the formula
$$
\int_{0}^{x} f(a)d_{q} a = \lim_{y\to 0} (1-q) \sum_{r\in \NN} \left( xq^{r}f(xq^{r}) - yq^{r}f(yq^{r})\right), 
$$
for all $x\in \RR$ and all $f$ in $\CC[a,a^{-1}]$.  

The construction of traces $\tau_s$ provides one with an alternative way of proving that  the principal comodule algebra $\cO(L_q(l;1,l))$ is not cleft. This involves evaluating $\tau_s$ at  the zero-component of the {\em Chern-Galois} character of the principal $\cO(U(1))$-comodule algebra $\cO(L_q(l;1,l))$.

As explained in \cite{DabGro:str} any strong connection in  a principal $H$-comodule algebra $A$ can be used to construct finitely generated projective modules over the coinvariant subalgebra $B$. To achive this one needs to take any finite dimensional left comodule $W$ over $H$ (or a finite dimensional corepresentation of $H$) and then the cotensor product $A\Box_HW$ is a finitely generated projective left $B$-module.  In the case of $\cO(L_q(l;1,l))$ such projective modules or {\em line bundles} over $\cO(\WW\PP_q(1,l))$ come out as
$$
\cL[n] := \{ x\in \cO(L_q(l;1,l)) \; |\; \varrho_l(x) = x\ot u^n\}, \qquad n\in \ZZ.
$$
An idempotent $E[n]$ for $\cL[n]$ is given as follows. Write $\omega(u^n) = \sum_i  \omega(u^n)\su 1 _i \ot \omega(u^n)\su 2_i$. Then 
\begin{equation}\label{projectors}
E[n]_{ij} = \omega(u^n)\su 2_i\omega(u^n)\su 1_j \in \cO(\WW\PP_q(1,l)).
\end{equation}
The traces of powers of each of the $E[n]$ make up a cycle in the cyclic complex  of $\cO(\WW\PP_q(1,l))$, whose corresponding class in  homology $HC_\bullet(\cO(\WW\PP_q(1,l)))$ is known as the Chern character of $\cL[n]$. In general, the result of the combined process that to an isomorphism class of a corepresentation of a Hopf algebra $H$ assigns the Chern character of the $B$-module associated to the $H$-principal comodule algebra $A$ is known as the {\em Chern-Galois character}. 
In particular, the traces of  $E[n]$ form the zero-component of the Chern-Galois character of $\cO(L_q(l;1,l))$. 
Should $\cO(L_q(l;1,l))$ be a cleft principal comodule algebra, then every module $\cL[n]$ would be free. Thus an alternative way of showing that  $\cO(L_q(l;1,l))$ is not cleft is to prove that, say, $\cL[1]$ is not a free left $\cO(\WW\PP_q(1,l))$-module. For this suffices it to show that $\tr E[1]$ is a non-trivial element of $HC_0(\cO(\WW\PP_q(1,l)))$ by proving that $\tau_s(\tr E[1])\neq 0$. First we compute $\tr E[1]$.

\begin{lemma}\label{lemma.cher}
The zero-component of the Chern character of $\cL[1]$ is the class of
\begin{equation}\label{cher}
\tr E[1] = 1 +\sum_{m=1}^l (-1)^mq^{m(m-1)}\left(1-q^{-2ml}\right)\binom  l m_{\! q^2} a^m .
\end{equation}
\end{lemma}
\begin{proof}
First, observe that the formula \eqref{binom} yields the following identity for q-binomial coeffcients:
\begin{equation}\label{binom.id}
\binom  l m_{\! x^{-1}} = x^{m(m-l)} \binom  l m_{\! x}.
\end{equation}
Next, remember that $dd^* = a\in \cO(\WW\PP_q(1,l))$. Having these observations at hand the rest is a straightforward calculation:
\begin{eqnarray*}
\tr E[1] & =& cc^* -\sum_{m=1}^l (-1)^mq^{-m(m+1)} {\binom  l m_{\! q^{-2}}}d^m {d^*}^{m} \\
&= &\prod_{m=0}^{l-1}(1-q^{2m}a) -\sum_{m=1}^l (-1)^mq^{2m(m-l)-m(m+1)} \binom  l m_{\! q^{2}}a^m \\
&= &1 + \sum_{m=1}^l (-1)^m\left(q^{m(m-1)}- q^{m(m-1) -2ml}\right) \binom  l m_{\!q^{2}}a^m \\
& = &1 +\sum_{m=1}^l (-1)^m q^{m(m-1)}\left(1-q^{-2ml}\right)\binom  l m_{\! q^2} a^m,
\end{eqnarray*}
where the first equality follows by \eqref{eq.strong1}, the second by \eqref{rel.lens2} and \eqref{binom.id}, while the third equality is a consequence of \eqref{binom}.
\end{proof}

\begin{proposition}\label{prop.pair}
For all $s=1,2,\ldots, l$, let $\tau_s$ be the cyclic cocycle on $\cO(\WW\PP_q(1,l))$ constructed in Proposition~\ref{prop.Fred}. Then $\tau_s(\tr E[1]) =1$. Consequently, the left $\cO(\WW\PP_q(1,l))$-module $\cL[1] $ is not free.
\end{proposition}
\begin{proof} Use \eqref{tau} and \eqref{cher} to calculate
\begin{eqnarray*}
\tau_s(\tr E[1]) &=& \sum_{m=1}^l (-1)^m q^{m(m-1)}\left(1-q^{-2ml}\right)\binom  l m_{\! q^2} \frac{q^{2ms}}{1-q^{2ml}}\\
&=& - \sum_{m=1}^l (-1)^{m} q^{m(m-1)}\binom  l m_{\! q^2} q^{2m(s-l)}\\
&=& 1- \sum_{m=0}^l (-1)^{m} q^{m(m-1)}\binom  l m_{\! q^2} q^{2m(s-l)}\\
&=& 1 - \prod_{m=1}^l(1-q^{2(s-l+m-1)}) =1 - \prod_{m=1}^l(1-q^{2(s-m)}) = 1.
\end{eqnarray*}
The last equality follows from the observation that since $s=1,2,\ldots ,l$, one of the factors in the product must vanish. The fourth equality is a consequence of the definition of q-binomial coefficients \eqref{binom}.
\end{proof}

\section{Algebras of continuous functions on quantum  teardrops and their $K$-theory}\label{sec.top} 
\setcounter{equation}{0}
The $C^*$-algebra $C(\WW\PP_q(k,l))$ algebra of continuous functions on the quantum weighted projective space $\WW\PP_q(k,l)$ is  defined as the universal $C^*$-algebra for relations \eqref{wpkl}. As a concrete algebra, $C(\WW\PP_q(k,l))$ is the subalgebra of bounded operators on the Hilbert space $\bigoplus_{s=1}^{l}V_s$ obtained as the completion of $\bigoplus_{s=1}^{l} \pi_s\left( \cO(\WP_q(k,l))\right)$; see Proposition~\ref{prop.rep}. In this section we show that this $C^*$-algebra  is isomorphic to the direct sum of compact operators with adjoined  identity.
\begin{proposition}\label{prop.seq}
Let $\cK_s$ denote the algebra of all compact operators on the Hilbert space $V_s$. There is a split-exact sequence of $C^*$-algebra maps
\begin{equation}\label{seq}
\xymatrix{0 \ar[r] & \bigoplus_{s=1}^{l}\cK_s \ar[r] & C(\WW\PP_q(k,l)) \ar[r] & \CC \ar[r] & 0.}
\end{equation}
\end{proposition}
\begin{proof} We use a method of proof similar to that of \cite[Proposition~1.2]{She:poi}. Write $\pi^\oplus$ for $\bigoplus_{s=1}^{l} \pi_s$. A basis for $\bigoplus_{s=1}^{l}V_s$ consists of eigenvectors $e_p^s$ of $\pi^\oplus(a)$ with distinct eigenvalues $q^{2(lp+s+1)}$. Since, for all $s$, $q^{2(lp+s+1)}\to 0$ as $p\to \infty$, $\pi^\oplus(a)$ is a compact operator. Similarly,  matrix coefficients $q^{lp+s+1} \prod_{r=0}^{l-1}\left(1 - q^{2(lp+s-r)}\right)^{1/2}$ of $\pi_s(b)$ tend to $0$ as $p$ tends to infinity, hence also $\pi^\oplus(b)$ is a compact operator; compare the proof of Proposition~\ref{prop.Fred}. This proves that the kernel of the projection of $C(\WW\PP_q(k,l))$ on the identity component $\CC$ contains only compact operators. 

The spectrum of  $\pi^\oplus(a)$ consists of distinct numbers
$$
\spec\left(\pi^\oplus(a)\right)= \{0\} \cup \{q^{2(lp+s)} \; |\; s=1,2,\ldots l,\; p\in \NN\};
$$ 
see  the proof of Proposition~\ref{prop.rep}. By functional calculus, for any $s$ and $p$  there are operators $f_{p,s}(\pi^\oplus(a))$ in $C(\WW\PP_q(k,l))$
 with spectrum given by  
$$
f_{p,s} : \spec\left(\pi^\oplus(a)\right) \to \CC, \qquad 0\mapsto 0, \qquad q^{2(ln+t)}\mapsto \delta_{s,t}\delta_{p,n}.
$$
Hence $C(\WW\PP_q(k,l))$ contains all orthogonal projections $P_p^s$ to one-dimensional spaces spanned by the $e_p^s$. More explicitly, these are obtained as limits:
\begin{equation}\label{pps}
q^{-2ns}\prod_{r=0,\; r\neq p}^n \frac{\pi_s(a)-q^{2(lr+s)}}{q^{2lp}-q^{2lr}}\xymatrix{ \ar[rr]_-{n\to \infty} && P_p^s.}
\end{equation}
Next, noting that $\pi_s(b)$ and $\pi_s(b^*)$ are shift-by-one operators with non-zero coefficients all the remaining generators of $\cK_s$ (and hence of $\bigoplus_{s=1}^{l}\cK_s$) can be obtained as products  of the rescaled   $\pi_s(b)$  and $\pi_s(b^*)$. Finally, the first map in the sequence \eqref{seq} is injective since all the $\pi_s$ are faithful representations.
\end{proof}

The following corollaries are straightforward consequences of Proposition~\ref{prop.seq}.

\begin{corollary}\label{cor.alg}
The $C^*$-algebra $C(\WW\PP_q(k,l))$ is isomorphic to the direct sum of $l$-copies of algebras of compact operators with the adjoined identity.
\end{corollary}

\begin{corollary}\label{cor.k}
The $K$-groups of $C(\WW\PP_q(k,l))$ are:
$$
K_0\left(C(\WW\PP_q(k,l))\right) = \ZZ^{l+1}, \qquad K_1\left(C(\WW\PP_q(k,l))\right) =0.
$$
\end{corollary}
\begin{proof}
This follows immediately from Proposition~\ref{prop.seq} by recalling that $K_0(\cK) = K_0(\CC) = \ZZ$ and $K_1(\cK) = K_1(\CC) = 0$, where $\cK$ is the $C^*$-algebra of compact operators on a separable Hilbert space.
\end{proof}

The first of the $\ZZ$ in $K_0\left(C(\WW\PP_q(k,l))\right) $ corresponds to the rank of free modules, the remaining ones are generated by the classes of projections $P^s_0$; see \eqref{pps}. The cyclic cocycles $\tau_s$ constructed in Proposition~\ref{prop.Fred} (see \eqref{tau}) extend to cocycles on $C(\WW\PP_q(k,l)$. Since, for any $x$ that is not a multiple of identity,  $\tau_s(x) = \tr (\pi_s(x))$ we immediately conclude that $\tau_s(P_0^s) = 1$, and the index pairing between the K-theory and cyclic cohomology of $C(\WW\PP_q(k,l)$ is given by
$$
\langle [\tau_s] , [P_0^t] \rangle = \delta_{s,t} \tau_s(P_0^t) = \delta_{s,t}.
$$

  \section*{Acknowledgments}
 The research of TB is  partially supported by the European Commission grant 
PIRSES-GA-2008-230836 and  the Polish Government grant 1261/7.PR UE/2009/7.


\begin{thebibliography}{99}{} 
\bibitem{AurKat:mir} D.\ Auroux, L.\ Katzarkov \&  D.\ Orlov, {\em Mirror symmetry for weighted projective planes and their noncommutative deformations}, Ann.\ Math.\  {\bf 167} (2008), 867--943.


\bibitem{Brz:syn} T.\ Brzezi\'nski, {\em On synthetic interpretation of quantum principal bundles}, AJSE D - Mathematics {\bf 35(1D)} (2010), 13--27.

\bibitem{BrzHaj:Che} T.\ Brzezi\'nski \& P.M.\ Hajac, {\em The Chern-Galois character},
Comptes Rendus Math. (Acad.\ Sci.\ Paris Ser.\ I) vol.\ 338 (2004) 113--116.

\bibitem{BrzMaj:gau}
{T.\ Brzezi\'nski \& S.\ Majid,} {\em Quantum group gauge theory on quantum spaces}, Comm.\ Math.\ Phys.\ {\bf 157} (1993), 591--638. Erratum 167 (1995), 235.


\bibitem{Con:non}
A.\ Connes, {\em Noncommutative Geometry,} Academic Press, New York 1994.

\bibitem{DabGro:str} { L.\ D\c{a}browski, H.\ Grosse \& P.M.\ Hajac},
{\em Strong connections and Chern-Connes pairing in the Hopf-Galois theory},
Comm.~Math.~Phys.\ {\bf 220} (2001), 301--331.

\bibitem{DAnLan:bou} F.\ D'Andrea \& G.\ Landi, {\em Bounded and unbounded Fredholm modules for quantum projective spaces}, J.\ K-Theory {\bf 6} (2010), 231--240.

\bibitem{Haj:str} P.M.~Hajac {\em Strong connections on quantum principal bundles}
Commun.\  Math.\ Phys.\ {\bf 182} (1996,), 579--617.

\bibitem{HajKra:pie} P.M.~Hajac, U.~Kr\"{a}hmer, R.~Matthes \& B.~Zieli\'nski {\em Piecewise principal comodule algebras}, arXiv:0707.1344, to appear in J.\ Noncommut.\ Geom.

\bibitem{HajMaj:pro} {P.M.\ Hajac \& S.\ Majid,} {\em Projective module description of the $q$-monopole}, Comm.\ Math.\ Phys.\ {\bf 206} (1999), 247--264.

\bibitem{HajMat:gra} P.M.\ Hajac, R.\ Matthes \& W.\ Szyma\'nski {\em Graph C*-algebras and $\ZZ_2$-quotients of
quantum spheres.} 
Rep.\ Math.\ Phys.\  {\bf 51} (2003), 215--224.

\bibitem{HawLan:fred} E.\ Hawkins \& G.\ Landi, {\em Fredholm modules for quantum Euclidean spheres}, J.\ Geom.\ Phys.\ {\bf 49} (2004), 272--293.

\bibitem{HonSzy:len} J.H.\ Hong \& W.\ Szyma\'nski, {\em Quantum lens spaces and graph algebras}, Pacific J.\ Math.\ 211 (2003), 249--263.


\bibitem{MasNak:dif} T.\ Masuda, Y.\ Nakagami \& J.\ Watanabe, {\em Noncommutative differential geometry on the quantum two sphere of Podle\'s. I. An algebraic viewpoint}, K-Theory {\bf 5} (1991), 151--175.

\bibitem{Mey:pro} U.\  Meyer, {\em Projective quantum spaces}, Lett.\ Math.\ Phys.\ {\bf 35} (1995),  91--97. 


\bibitem{Sch:pri}  H.-J.\ Schneider, {\em Principal homogeneous spaces for
        arbitrary Hopf algebras}, Israel J.\ Math.\ {\bf 72} (1990), 167--195. 

\bibitem{She:poi} A.J.-L.\ Sheu {\em Quantization of the Poisson SU(2) and its Poisson homogeneous space -- the 2-sphere}, Comm.\ Math.\ Phys.\ {\bf 135} (1991), 21--232.

\bibitem{VakSob:su2} Ya.S.\ Soibel'man   \&  L.L.\ Vaksman, {\em An algebra of functions on the quantum group SU(2),} (Russian) Funktsional.\ Anal.\ i Prilozhen.\ {\bf 22} (1988),  1--14, 96; translation in Funct.\ Anal.\ Appl.\ {\bf 22} (1988),  170--181.


\bibitem{VakSob:alg}Ya.S.\ Soibel'man   \&  L.L.\ Vaksman, {\em 
Algebra of functions on the quantum group SU(n+1), and odd-dimensional quantum spheres.} (Russian) Algebra i Analiz {\bf 2} (1990), 101--120; translation in 
Leningrad Math.\ J.\ {\bf 2} (1991),  1023�1042.

\bibitem{Thu:thr}  W.P.\ Thurston, {\em Three-dimensional geometry and topology,}  Edited by Silvio Levy. Princeton Mathematical Series, 35, Princeton University Press, Princeton, NJ, (1997).


\bibitem{Wor:twi}
S.L.\ Woronowicz, {\em Twisted ${\rm SU}(2)$ group. An example of a 
noncommutative differential calculus.} Publ.\ Res.\ Inst.\ Math.\ 
Sci.\ {\bf 23} (1987),
117--181.


\bibitem{Wor:com} S.L.\ Woronowicz {\em Compact matrix pseudogroups} Comm.~Math.~Phys. {\bf 111} (1987), 613--665.

\end{thebibliography}
\end{document}